\documentclass[12pt]{amsart}
\usepackage[margin=1.2in]{geometry}

\usepackage{amsmath,amsthm,amssymb,bbm,mathtools}
\usepackage{wrapfig}
\usepackage{subcaption}

\usepackage[usenames]{color}

\newcommand{\black}{\black}

\usepackage[font=small,labelfont=bf]{caption}
\usepackage[bookmarks, colorlinks, breaklinks,
pdftitle={},pdfauthor={}]{hyperref}

\hypersetup{
  linkcolor=black,
  citecolor=black,
  filecolor=black,
  urlcolor=black}

\newtheorem*{theorem*}{Theorem}
\newtheorem{theorem}{Theorem}
\newtheorem{lemma}[theorem]{Lemma}
\newtheorem{corollary}[theorem]{Corollary}
\newtheorem{defn}[theorem]{Definition}
\newtheorem*{defn*}{Definition}
\newtheorem{prop}[theorem]{Proposition}
\newtheorem*{prop*}{Proposition}

\newtheorem*{conjecture*}{Conjecture}

\newtheorem*{fact*}{Fact}

\newtheorem*{remark}{Remark}
\numberwithin{equation}{section}

\newcommand{\sss}[1]{\scriptscriptstyle{#1}}

\usepackage{tikz}

\usepackage{booktabs}

\usetikzlibrary{calc,decorations.markings}
\usetikzlibrary{decorations.pathmorphing}
\usetikzlibrary{decorations.pathreplacing}
\usetikzlibrary{patterns}
\usetikzlibrary{arrows}

\newcommand{\bb}[1]{\mathbb{#1}}
\newcommand{\cc}[1]{\mathcal{#1}}

\newcommand{\ob}[1]{\left(#1\right )} 
\newcommand{\cb}[1]{\left[#1\right ]} 
\newcommand{\ab}[1]{\langle #1 \rangle} 
\newcommand{\abs}[1]{\left\vert#1\right\vert} 

\newcommand{\V}[1]{\boldsymbol{#1}}
\newcommand{\blank}[1]{}
\newcommand{\bydef}{\equiv}

\newcommand{\E}{\bb E}

\newcommand{\Z}{\bb Z}

\newcommand{\N}{\bb N}
\renewcommand{\P}{\bb P}

\newcommand{\LE}{\mathsf{LE}}
\newcommand{\saw}{\mathsf{SAW}}
\newcommand{\fe}{\mathrm{dim}_{\mathrm{H}}}

\title{Loop-erased random walk as a spin system observable}
\author{Tyler Helmuth}
\author{Assaf Shapira} 

\address{Department of Mathematical Sciences, Durham University, Lower Mountjoy, DH1 3LE Durham, United Kingdom. Email: tyler.helmuth@durham.ac.uk.}
\address{Universit\`a Roma Tre. Email: assaf.shapira@normalesup.org.}
\begin{document}

\begin{abstract}
  The determination of the Hausdorff dimension of
  the scaling limit of loop-erased random
  walk is closely related to the study of the one-point function
  of loop-erased random walk, i.e., the probability a loop-erased
  random walk passes through a given vertex. Recent work in the
  theoretical physics literature has investigated the Hausdorff
  dimension of loop-erased random walk in three dimensions by applying
  field theory techniques to study spin systems that heuristically
  encode the one-point function of loop-erased random walk. Inspired
  by this, we introduce two different spin systems whose correlation
  functions can be rigorously shown to encode the one-point function
  of loop-erased random walk.
\end{abstract}

\maketitle

\section{Introduction}
\label{sec:introduction}

\emph{Loop-erased random walk} is, informally speaking, the
probability measure on self-avoiding walks that results from removing
the loops from simple random walk in chronological order. We give a
precise description in Section~\ref{sec:LERW} below. Loop-erased
random walk is a fundamental probabilistic object with connections to
spanning trees and the uniform spanning forest~\cite{Wilson,
  Pemantle}, amongst other topics. In two dimensions, loop-erased
random walk has $\mathrm{SLE}_{2}$ as a scaling
limit~\cite{LawlerSchrammWerner}, while in four and higher dimensions
it scales to Brownian motion~\cite{Lawler}. It is possible to prove
the scaling limit of loop-erased random walk exists in three
dimensions~\cite{Kozma}, but many open questions remain,
see~\cite{LiShiraishi,AngelEtAl} and references therein.

The preceding results have been used to determine the fractal
(Hausdorff) dimension $\fe(\cc K_{d})$ of the scaling limit
$\cc K_{d}$ of $d$-dimensional loop-erased random walk when $d\neq 3$:
$\fe(\cc K_{2})=\frac{5}{4}$ and $\fe(\cc K_{d})=2$ for $d\geq
4$. Shiraishi~\cite{ShiraishiHausdorff} has given a characterization
of $\fe(\cc K_{3})$, but the numerical value is not known
rigorously. These results are all based on probabilistic tools.

Loop-erased random walk is also of interest within theoretical
physics. An interesting recent development has been the use of
non-rigorous field theory techniques for the determination of the
Hausdorff dimension
$\fe(\cc K_{d})$~\cite{FedorenkoDoussalWiese,FedorenkoWiese}. While
there is a long history of the interplay between field theories and
random walks~\cite{Symanzik,BrydgesFrohlichSpencer,BHS2}, geometric
properties of loop-erased random walk are less obviously connected to a
field theory due to `erasure' in its definition.

In this note we describe two rigorous spin system representations of
loop-erased random walk. Both representations translate the problem of
the determination of the Hausdorff dimension for loop-erased random
walk into a discrete spin system problem. This can be viewed as a
mathematical justification for the starting point of the non-rigorous
field theory steps contained in~\cite{FedorenkoWiese}, albeit for
somewhat different spin systems than the one considered
in~\cite{FedorenkoWiese}. Our proofs use a combination of Grassmann
integration and Viennot's combinatorial theory of heaps of
pieces~\cite{Abdesselam, Viennot}, and pass through intermediary
representations of our spin systems in terms of a graphical loop
models.

We defer a precise descriptions of our spin systems to
Section~\ref{sec:precise} below, but remark here that they
  both contain two bosonic components and four fermionic components,
and as a result a weight of $2-4=-2$ for each loop. This is natural,
as various ``$O(-2)$'' models have been connected with simple random
walk in the past in
physics~\cite{AbeHatano1974,BalianToulouse1973,Fisher1973,Knops1973}
and combinatorics and probability~\cite{Viennot,HelmuthLWW,Marchal}.
As highlighted in~\cite{FedorenkoWiese}, the extra components of a
spin system involved in writing $-2=2-4$ enables one to capture
loop-erased random walk statistics in addition to simple random walk
statistics.

To give a flavour of our results, let $B_{R}(0)$ denote the ball of
radius $R$ about the origin in $\Z^{d}$, and let $\partial B_{R}(0)$
denote the vertex boundary of the ball. For three distinct vertices $a,b$ and $c$
let $U_{m^{2}}(a,b,c)$ denote
either of the three-point correlation functions introduced below in
Section~\ref{sec:precise} on $B_{R}(0)$ with vertex weights $m^{2}>0$
on $\partial B_{R}(0)$.
\begin{theorem}
  \label{thm:Main-informal}
  Let $\LE(\omega)$ denote the loop-erased random walk trajectory that
  results from simple random walk started from $0\in B_{R}(0)$ and
  stopped at $\partial B_{R}(0)$. Then for any vertex $b\in B_{R-1}(0)\setminus \{0\}$,
  \begin{equation}
    \label{eq:Main-informal}
    \P\cb{b\in \LE(\omega)} = 
 \lim_{m^{2}\to\infty} \sum_{c\in \partial B_{R}(0)} U_{m^{2}}(0,b,c).
  \end{equation}
\end{theorem}
The left-hand side in~\eqref{eq:Main-informal} is the one-point
function for loop-erased random walk. The scaling behaviour of the
one-point function encodes $\fe(\cc K_{d})$, see~\cite{LiShiraishi1pt}
for a precise statement when $d=3$. Theorem~\ref{thm:Main-informal} is
thus our promised translation of the problem of determining
$\fe(\cc K_{d})$ into a spin system problem.

A non-rigorous analysis of a spin system similar to the one of this
paper has been used to estimate
$\fe(\cc K_{3})\approx 1.62$~\cite{FedorenkoWiese}, which agrees with
extensive simulations~\cite{WilsonD}. Rigorously
establishing a similar result would be extremely interesting.  Another
very interesting direction would be to investigate if there are
further connections between our spin systems and $\mathrm{SLE}_{2}$ in
two dimensions. We comment on some further future directions below in
Section~\ref{sec:main-result}. 

\subsection*{Acknowledgements}
\label{sec:acknowledgements}

The authors thank K.\ J.\ Wiese for bringing this problem to our
attention. We also thank the referees for their thorough and helpful
reports on a previous version of this article. TH was at the
University of Bristol when this work was carried out, and was
supported by EPSRC grant EP/P003656/1.

\section{Precise formulation of the result}
\label{sec:precise}

Our main result holds for any finite connected graph
$\cc G = (\cc V, \cc E)$.  The precise formulation relies on the use
of Grassmann algebras and Grassmann integration. The reader unfamiliar
with this subject can consult, e.g.,~\cite[Section~2]{Abdesselam}.

\subsection{The spin systems}
\label{sec:spin-system}

We first introduce the spin systems used in our representation. This
requires some notation. Let $\cc V$ be a finite set. Consider the
Grassmann algebra whose generators are
$\{\xi^{\sss (i)}_{x}, \eta^{\sss (i)}_{x}\}_{x\in \cc V, i=1,\dots,
  4}$. That is, these variables are all pairwise anticommuting.  We
set $\phi_{x}= \eta^{\sss (3)}_{x}\xi^{\sss (3)}_{x}$ and
$\psi_{x}= \eta^{\sss (4)}_{x}\xi^{\sss (4)}_{x}$; these are commuting
elements of the algebra. Note that $\phi_{x}^{2}=\psi_{x}^{2}=0$ for
all $x$.

The spins of our spin systems are six-tuples $\sigma_{x}$ for
$x\in \cc V$. For notational reasons that will become clear in what
follows, we write the six-tuples as triples of pairs
$\sigma^{\sss (i)}_{x} = (u^{\sss (i)}_{x},v^{\sss (i)}_{x})$, for
$i=1,2,3$.  The $u^{\sss (i)}$ and $v^{\sss (i)}$ are given by
$(u^{\sss (i)}_{x},v^{\sss (i)}_{x}) = (\eta^{\sss (i)}_{x},\xi^{\sss
  (i)}_{x})$ for $i=1,2$ and
$(u^{\sss (3)}_{x},v^{\sss (3)}_{x}) = (\phi_{x},\psi_{x})$. We define
a product
\begin{equation}
  \label{eq:IP}
  \sigma_{x}\cdot \sigma_{y} \bydef \sigma_{x}^{\sss (1)}\cdot \sigma_{y}^{\sss (1)} +
  \sigma_{x}^{\sss (2)}\cdot \sigma_{y}^{\sss (2)} + \sigma_{x}^{\sss (3)}\cdot \sigma_{y}^{\sss (3)},
\end{equation}
where
  \begin{equation}
    \label{eq:IP-new}
    \sigma_{x}^{\sss (i)}\cdot \sigma_{y}^{\sss (i)} \bydef
    u_{x}^{\sss (i)}v_{y}^{\sss (i)} + u_{y}^{\sss (i)}v_{x}^{\sss (i)}.
  \end{equation}

Let $\beta = (\beta_{xy})_{x,y\in \cc V}$ be a collection of
non-negative and symmetric edge weights, i.e.,
$\beta_{xy} =\beta_{yx}\geq 0$, and $\beta_{xx}=0$ for all
$x$. Further, let $(m_{x}^{2})_{x\in \cc V}$ be a collection of
  non-negative vertex weights, and define $r_{x}\bydef m_{x}^{2} + \sum_{y}\beta_{xy}$.

\subsubsection{Symmetric action}
\label{sec:S}

For notational convenience we introduce the so-called
$\tau$-field on pairs of vertices $x\neq y$, cf.~\cite{BBSbook}:
\begin{equation}
  \label{eq:tau}
  \tau^{\sss (i)}_{xy} \bydef \beta_{xy}\sigma^{\sss (i)}_{x}\cdot \sigma^{\sss
    (i)}_{y}, \qquad i=1,2,3,
\end{equation}
and we also introduce the shorthand
\begin{equation}
  \label{eq:tau-div}
  (\nabla\cdot\tau)_{x}^{2} \bydef  \sum_{i=1}^{3}\ob{\sum_{y\in\cc V} \tau^{\sss (i)}_{xy}}^{2}.
\end{equation}

Define an action $S$ by 
\begin{equation}
  \label{eq:action}
  S \bydef
  \frac{1}{2} \sum_{x\in \cc V}\ob{ 
 r_{x} \sigma_{x}\cdot \sigma_{x} + \sum_{y}(\tau_{xy}^{\sss (3)}-\tau_{xy}^{\sss
      (1)}-\tau_{xy}^{\sss (2)} ) - \frac{1}{2}((\nabla\cdot
  \tau)^{2}_{x})^{2} + \frac{1}{12}((\nabla\cdot\tau)_{x}^{2})^{3}}.
\end{equation}

Note that $S$ is an even (i.e., commuting) element of the Grassmann
algebra, so its exponential is well-defined. We define the partition function $Z$ by
\begin{equation}
  \label{eq:modelZ}
  Z \bydef \int \prod_{x\in \cc V}\ob{ 
    \prod_{i=1}^{4}
 \partial_{\xi^{\sss (i)}_{x}}\partial_{\eta^{\sss (i)}_{x}}}\,  e^{S}.
\end{equation}
For $F$ an element of the Grassmann algebra
we define a normalized expectation by 
\begin{equation}
  \label{eq:model}
  \ab{F} \bydef \frac{1}{Z} \int \prod_{x\in \cc V} 
  \ob{ \partial_{\xi^{\sss (i)}_{x}}\partial_{\eta^{\sss (i)}_{x}}} \, F e^{S}.
\end{equation}
This is a rational function of the edge and vertex variables, and in
our cases of interest it will be clear that the evaluation of this
rational function is finite.

Our main observable of interest will be, for distinct $a,b,c\in \cc V$,
\begin{equation}
  \label{eq:obs}
  U(a,b,c) \bydef 
  C_{bc} \ab{ u^{\sss (2)}_{c}v^{\sss (2)}_{b}
      u^{\sss (3)}_{b}v^{\sss (3)}_{b}
      u^{\sss (1)}_{b}v^{\sss (1)}_{a}}, \quad C_{bc} \bydef  m_{c}^{2} 
    (m_{b}^{2}+\sum_{y\in\cc V}\beta_{by})^{2}.
\end{equation}

\begin{remark}
  \label{rem:Boson-IP}
  The products $\phi_{x}\psi_{y}$ and $\psi_{x}\phi_{y}$ in \eqref{eq:IP-new}
  can be replaced with the symmetric expression
  $\phi_{x}\phi_{y}+\psi_{x}\psi_{y}$ on bipartite graphs by
  exchanging the roles of $\phi$ and $\psi$ on one bipartition. On
  non-bipartite graphs there is also natural symmetrization procedure,
  but the notation of~\eqref{eq:IP-new}
  will be more convenient for what
  follows.
\end{remark}

\begin{remark}
We have formulated the action~\eqref{eq:action} in a combinatorially
convenient manner. A more conventional form from the viewpoint of
field theory can be obtained by algebraic manipulations, see~\cite{ShapiraWiese}.
\end{remark}

\subsubsection{Chiral action}
\label{sec:S'}

\newcommand{\chiraltau}{\tau'}

For our second representation we use a variant of the term
$\nabla\cdot\tau$ from Section~\ref{sec:S}. Define
\begin{equation}
  \chiraltau_{x} \bydef \sum_{i=1}^{3}\sum_{y\in\cc
    V}\beta_{xy}v^{\sss (i)}_{y}u^{\sss (i)}_{x},
\end{equation}
and define the action
\begin{equation}
  \label{eq:action-2}
  S'\bydef \frac{1}{2} \sum_{x} \left(
    r_x \sigma_x \cdot \sigma_x 
    + 2 \tau'_x
    - (\chiraltau_{x})^2
    + \frac{2}{3} (\chiraltau_{x})^3
  \right).
\end{equation}

The observable of interest for the action $S'$ is, for $a,b,c\in\cc V$ distinct,
\begin{equation}
  \label{eq:U-2}
  U'(a,b,c) \bydef C'_{bc}\ab{u_{c}^{\sss (2)}v_{b}^{\sss
      (2)}u_{b}^{\sss (1)}v_{a}^{\sss (1)}}', \qquad C'_{bc}\bydef
  m_{c}^{2}(m_{b}^{2}+\sum_{y\in\cc V}\beta_{by}),
\end{equation}
where $\ab{\cdot}'$ is the normalized expectation defined by replacing
$S$ by $S'$ in~\eqref{eq:modelZ}--\eqref{eq:model}.

\subsection{Loop-erased random walk}
\label{sec:LERW}

Let $\beta = (\beta_{xy})_{x,y\in \cc V}$ and
$(m_{x}^{2})_{x\in \cc V}$ be as Section~\ref{sec:spin-system}. The set of edges
$\cc E$ with strictly positive weights induces a graph
$\cc G = (\cc V, \cc E)$. We will assume that the graph $\cc G$ is
connected, and that $m_{x}^{2}>0$ for some $x\in \cc V$.

Let $\blacktriangle\notin \cc V$ be an additional `cemetery'
vertex. We define a discrete-time Markov chain $X$ with state space
$\cc V \cup \{\blacktriangle\}$, $\blacktriangle$ an absorbing state, by setting
\begin{equation}
  \label{eq:RW}
  \P\cb{X_{n+1}=y \mid X_{n}=x} = 
  \begin{cases}
  \frac{\beta_{xy}}{m_{x}^{2}+\sum_{y}\beta_{xy}} & y\neq \blacktriangle, \\
  \frac{m_{x}^{2}}{m_{x}^{2}+\sum_{y}\beta_{xy}} & y=\blacktriangle.
  \end{cases}
\end{equation}
Henceforth we will simply refer to $X$ as a simple random walk, and we
view the walk as a sequence of nearest-neighbour vertices of $\cc G$.

To formally define loop-erasure, some definitions are needed. Given a
\emph{length} $k\geq 1$ and a finite walk
$\omega = (\omega_{1},\omega_{2},\dots,\omega_{k})$ we let
$\abs{\omega}=k$. A walk is \emph{self-avoiding} if
$\omega_{i}=\omega_{j}$ implies $i=j$. A \emph{rooted and directed
  cycle} is a walk with $\omega_{1}=\omega_{\abs{\omega}}$ such that
$(\omega_{1},\dots, \omega_{\abs{\omega}-1})$ is a non-empty
self-avoiding walk. Sometimes this is called a self-avoiding polygon,
but note our definition includes self-avoiding polygons of length
three that use the same edge twice. 
A \emph{directed cycle} is an
equivalence class of rooted and oriented cycles, the equivalence being
under cyclic shifts.

Given a walk $\omega$, suppose $K$ is the first index such that
$(\omega_{1},\dots, \omega_{K})$ is not a self-avoiding walk. If $K$
is finite, let $K'<K$ be the unique index such that
$(\omega_{K'},\dots, \omega_{K})$ is a rooted oriented cycle. Define
$L(\omega) = (\omega_{1},\dots, \omega_{K'},\omega_{K+1},\dots,
\omega_{\abs{\omega}})$ if $K$ is finite, and $L(\omega)=\omega$
otherwise. Note $L(\omega)$ is a walk. For any finite walk $\omega$ we define the
\emph{loop-erasure $\LE(\omega)$ of $\omega$} to be the walk that
results from iteratively applying $L$. This operation stabilizes
after finitely many steps since each application of $L$ to a walk that
is not self-avoiding reduces the length of the walk by at least one.

Let $h_{\blacktriangle}$ be the (almost surely finite) hitting time of
$\blacktriangle$, and let $h_{\blacktriangle}^{-} = h_{\blacktriangle}-1$.
\emph{Loop-erased random walk} is the law of
$\LE( (X_{n})_{n< h_{\blacktriangle}})$, i.e., the law the loop-erasure of $X$
considered up until the time of the first jump to $\blacktriangle$.

\subsection{Main results}
\label{sec:main-result}

Our first result concerns the action $S$. Recall that we assume
  throughout the paper that $\cc G$ is a finite graph.
\begin{theorem}
  \label{thm:main}
  Assume $\cc G$ is connected. For three distinct vertices $a,b,c\in \cc V$,
  \begin{equation}
    \label{eq:main}
    U(a,b,c) = \P_{a}\cb{\text{$b\in \LE((X_n)_{n<h_\blacktriangle})$ and
      }X_{h^{-}_{\blacktriangle}}=c}\cdot(1+O(m_c^{-2})) \quad \text{as
    } m_{c}^{2}\to\infty.
  \end{equation}
\end{theorem}

We obtain a similar result for the action $S'$, but without any need
to take $m_{c}^{2}\to\infty$.
\begin{theorem}
  \label{thm:main-2}
  Assume $\cc G$ is connected. For three distinct vertices
  $a,b,c,\in\cc V$ such that $m_{c}^{2}>0$,
  \begin{equation}
    \label{eq:main'}
    U'(a,b,c) = \P_{a}\cb{\text{$b\in \LE((X_n)_{n<h_\blacktriangle})$ and
      }X_{h^{-}_{\blacktriangle}}=c}.
  \end{equation}
\end{theorem}

Theorem~\ref{thm:Main-informal} follows from the preceding results,
since taking the killing rate $m_{x}^{2}$ to infinity for
$x\in A\subset \cc V$ results in a random walk stopped on the set
$A$. We give more details below after a brief discussion of these
theorems. 

By considering spin systems with additional components it seems
possible to construct observables that encode the multipoint functions of
loop-erased random walk. Variations on our formulas that replace the
`bosonic' variables $u^{\sss (3)}$ and $v^{\sss (3)}$ with standard
bosons also appear possible. We remark that bosonic variables with
square $0$ such as $u^{\sss (3)}$ and $v^{\sss (3)}$ can be viewed as
a way to implement a Nienhuis-type action~\cite{Nienhuis} on general
graphs.

We briefly compare the actions $S$ and $S'$. The spin system defined
by the action $S$ is closer to what is studied
in~\cite{FedorenkoWiese} than the spin system defined by the action
$S'$, as $S$ manifestly inherits the symmetries of the
products
in~\eqref{eq:IP-new}. 
On the other hand, the
exact identity of Theorem~\ref{thm:main-2} arises in part from the
asymmetry in the action $S'$. From the point of view studying
$\fe(\cc K_{d})$ one would presumably like as simple of an action and
observable as possible.  Based on the non-rigorous methods and results
of~\cite{FedorenkoWiese} one is lead to conjecture that a modification
of the action $S$ that removes the six-body term ($S_{6}$ below) would
lead to the same behaviour.  A justification of this conjecture would
be quite interesting. Whether the six-body term in $S'$ is similarly
negligible is not clear to us.

\begin{proof}[Proof of Theorem~\ref{thm:Main-informal}]
  We consider the action $S$; for $S'$ the argument is very
    similar. Let $z\in \partial B_{R}(0)$, and let $(X_{n})$ be a
    random walk on $B_{R}(0)$. It suffices to prove that
  \begin{equation}
    \lim_{m^{2}\to\infty} \P_{0,m^{2}}^{}\cb{y\in \LE((X_{n})_{n<h_{\blacktriangle}})\text{ and
        $X_{h_{\blacktriangle}^{-}}=z$}}
    =
    \P_{0}\cb{y\in \LE((X_{n})_{n\leq h_{z}})},
  \end{equation}
  where the left-hand side is the probability for a random walk killed
  on the boundary with rate $m^{2}$ and the right-hand side is the
  probability for a random walk without killing.  This can be seen by
  coupling the random walks together until they first hit the
  boundary. The probability that $y$ is in the loop-erasure of the
  stopped process is then precisely the probability that $y$ is in the
  loop-erasure of the killed process, conditionally on the killed
  process being killed at the first visit to the boundary.

  Since the probability the killed random walk is killed at its first
  visit to the boundary tends to one as $m^{2}\to\infty$, the claim
  follows from Theorem~\ref{thm:main}.
\end{proof}

The remainder of the paper is organized as follows. In
Section~\ref{sec:proof-theorem} we prove
Theorem~\ref{thm:main}. Theorem~\ref{thm:main-2} has a similar but
somewhat simpler proof, which is given in Section~\ref{sec:thm2}.

\section{Proof of Theorem~\ref{thm:main}}
\label{sec:proof-theorem}

A \emph{coloured graph} is a graph whose edges are each assigned a
non-empty subset of the colours $\{1,2,3\}$. We do not permit
  graphs to have self-loops (i.e., edges $\{x,x\}$), though we will
  make use of self-loops in the proofs below. A connected component is
\emph{monochromatic} if the edges of the component are all assigned
exactly one colour. Given a coloured graph $G$, a vertex $x$, and a
colour $i$, we say that the \emph{coloured vertex} $(x,i)$ is present if
there is vertex $y$ such that the edge $\{x,y\}$ is coloured $i$ in
$G$.

We remark that it is possible to view coloured graphs as multigraphs
in which each repeated edge is assigned a distinct colour. 

\subsection{Graphical representation of the partition function}
\label{sec:comp-part-funct}

For a coloured graph $G$ whose connected components are either
self-avoiding walks or monochromatic cycles, define
\begin{equation}
  \label{eq:Z-weight}
  w_{0}(G) \bydef (-1)^{F(G)}\prod_{xy\in G}\beta_{xy} \prod_{x,i\notin V(G)}r_{x},
\end{equation}
where $F(G)$ is the number of cycles in $G$ coloured either
$1$ or $2$, the first product is over all edges in $G$, and the last
product is over all coloured vertices not in $G$.
\begin{prop}
  \label{prop:Z-graph}
  The partition function~\eqref{eq:modelZ} associated to $\cc G = (\cc
  V, \cc E)$ is given by
  \begin{equation}
    \label{eq:Z-graph}
    Z = \sum_{G}w_{0}(G),
  \end{equation}
  where the sum is over directed coloured subgraphs $G$
  of $\cc G$ whose connected
  components are monochromatic directed cycles.
\end{prop}

The proof will make use of terms $S_{0}$, $S_{1}$, $S_{4}$ and $S_{6}$ defined by
\begin{align}
  \label{eq:expl-1}
  S_{0}(i;x) &\bydef r_{x}u^{\sss (i)}_{x}v^{\sss (i)}_{x},\\
  S_{1}(i;x,y) &\bydef \beta_{xy}v^{\sss (i)}_{y}u^{\sss (i)}_{x}, \\
  S_{4}(i,j;x,\V{y}) 
             &\bydef \ob{\prod_{k=1}^{4}\beta_{xy_{k}}}\,
               v^{\sss (i)}_{y_{1}}u^{\sss (i)}_{x}
               v^{\sss (i)}_{x}u^{\sss (i)}_{y_{2}}
               v^{\sss (j)}_{y_{3}}u^{\sss (j)}_{x}
               v^{\sss (j)}_{x}u^{\sss (j)}_{y_{4}},\\
  S_{6}(x,\V{y}) &\bydef  \ob{\prod_{k=1}^{6}\beta_{xy_{k}}}\,
                   v^{\sss (1)}_{y_{1}}u^{\sss (1)}_{x}
                   v^{\sss (1)}_{x}u^{\sss (1)}_{y_{2}}
                   v^{\sss (2)}_{y_{3}}u^{\sss (2)}_{x}
                   v^{\sss (2)}_{x}u^{\sss (2)}_{y_{4}}
                   v^{\sss (3)}_{y_{5}}u^{\sss (3)}_{x}
                   v^{\sss (3)}_{x}u^{\sss (3)}_{y_{6}},
\end{align}
where $\V{y}$ represents four- and six-tuples of vertices in $\cc V$
in the last two displays, respectively, while $i$ and $j$ range over
$\{1,2,3\}$. In what follows it will be contextually clear when
$\V{y}$ is a four- or six-tuple.
\begin{lemma}
  \label{lem:expl}
  The action can be rewritten as
  \begin{equation}
    \label{eq:expl}
    S = \sum_{x,i} S_{0}(i;x)
    + \sum_{x,y,i}S_{1}(i;x,y)
    - \sum_{x,\V{y},i,j}S_{4}(i,j;x,\V{y})
    + 2\sum_{x,\V{y}}S_{6}(x,\V{y}).
  \end{equation}
\end{lemma}
\begin{proof}
  This is an algebraic reformulation of~\eqref{eq:action}.
\end{proof}

\begin{lemma}
  \label{lem:expS}
  The exponential of the action can be rewritten as
  \begin{equation}
    \label{eq:factor}
    e^{S} = \prod_{x, i}(1+S_{0}(i;x))
    \prod_{x,y,i}(1+S_{1}(i;x,y)) 
    \prod_{x,\V{y},i,j}(1-S_{4}(i,j;x,\V{y}))
    \prod_{x,\V{y}}(1+2S_{6}(x,\V{y})).
  \end{equation}
\end{lemma}
\begin{proof}
  Each summand $\tilde S$ in~\eqref{eq:expl} is nilpotent. In particular,
  $\exp(\tilde S)=1+\tilde S$. Since each $\tilde S$ is even the exponential of
  the sum is the product of the exponentials.
\end{proof}

\begin{proof}[Proof of Proposition~\ref{prop:Z-graph}]
  We represent the terms in the expansion of the products
  in~\eqref{eq:factor} as coloured graphs with self-loops by
  viewing each monomial $v^{\sss (i)}_{y} u^{\sss (i)}_{x}$ as a
  directed edge with colour $i$ from $x$ to $y$. If $x=y$ this is a
  self-loop. Explicitly,
  \begin{itemize}
  \item terms $S_{0}(i;x)$ give self-loops $(x,x)$ of colour $i$,
  \item terms $S_{1}(i;x,y)$ give directed edges $(x,y)$ of colour
    $i$,
  \item terms $S_{4}(i,j;x,\V{y})$ give directed edges $(x,y_{1})$ and
    $(y_{2},x)$ of colour $i$ and $(x,y_{3})$ and $(y_{4},x)$ of
    colour $j$,
  \item terms $S_{6}(x,\V{y})$ give directed edges $(x,y_{1})$ and
    $(y_{2},x)$ of colour $1$, $(x,y_{3})$ and $(y_{4},x)$ of colour
    $2$, and $(x,y_{5})$, $(y_{6},x)$ of colour $3$.
  \end{itemize}

  By definition, $\int e^{S}$ is the coefficient of the top degree
  term of $e^{S}$. By the above, this coefficient is a sum over
  coloured directed graphs $G$ that use every vertex of $\cc G$ with
  every possible colour. Note that such a graph $G$ can arise in
  multiple ways, i.e., from different choices of terms in the
  expansions of the products in~\eqref{eq:factor}. For any such choice
  the weight of the graph has a factor $r_{x}$ for each self-loop at
  $x$; a factor $\beta_{xy}$ for each coloured directed edge $(x,y)$;
  a numerical factor determined by the coefficients of the $S_{i}$;
  and a sign determined by the re-ordering of the Grassmann
  variables. The last two considerations depend on the choices of
  terms in the expansions of the products in~\eqref{eq:factor}.

  The next step is to characterize the graphs $G$ that give a non-zero
  contribution. It suffices to characterize the contributing graphs up to
  self-loops, whose existence is implied by the fact that only top
  degree terms contribute to the integral. In what follows we
  implicitly discuss only edges that are not self-loops. We first
  prove that all graphs with non-zero weight are unions of disjoint
  cycles. It suffices to establish that
  \begin{enumerate}
  \item the coloured in- and out-degree of each vertex in $G$ is the same in
    any graph with non-zero weight,
  \item if the in-degree of some vertex in $G$ is two, then the weight is zero,
  \item if the in-degree of some vertex in $G$ is three, then the weight is zero.
  \end{enumerate}
  The first claim is immediate since self-loops add coloured in- and
  out-degree equally.

  For the second claim, fix a graph $G$ and a vertex $x$ of in- and
  out-degree two. Let $\V{y}$ be the 4-tuple of vertices to which $x$
  connects. Two of the edges must be of colour $i$, and two of colour
  $j$. There are two ways (fixing all other choices) to create the
  graph $G$ that differ in how the monomials associated to edges
  containing $x$ are chosen. The possibilities are
  $S_{1}(i;x,y_{1})S_{1}(i;x,y_{2}) S_{1}(j;x,y_{3}) S_{1}(j;x,y_{4})
  $ or $-S_{4}(i,j;x,\V{y})$, which sum to zero.

  For the third claim, we argue as above. There are now five
  possibilities for how the edges connecting $x$ to the vertices in
  the six-tuple $\V{y}$ could be chosen:
  \begin{itemize}
  \item a product of six $S_{1}(i;x,y)$,
  \item a product of one $-S_{4}(i,j;x,\V{y}')$ with two $S_{1}(k;x,y)$
    for $y\in \V{y}\setminus \V{y}'$ and $k = \{1,2,3\}\setminus
    \{i,j\}$,
  \item a factor $S_{6}(x,\V{y})$.
  \end{itemize}
  The first option contributes weight $1$, the last weight $+2$, while
  the middle three contribute weight $-1$. Thus these possibilities
  sum to zero.

  To complete the proof of the proposition, we note that each
  self-loop has weight $r_{x}$, and there is one such factor for each
  vertex $x$ and each colour $i$ that is not contained in some
  cycle. Thus all that remains is to justify the sign
  $(-1)^{F(G)}$. Each directed cycle $(\omega_1,\dots,\omega_k)$
  is associated to a product
  $S_1(i;\omega_1,\omega_2)S_1(i;\omega_2,\omega_3) \dots
  S_1(i;\omega_k,\omega_1)$, and each individual $S_{i}$ term
  commutes, i.e., letting $\omega_{k+1}=\omega_{1}$,
  \begin{equation}
    S_1(i;\omega_k,\omega_{k+1}) \dots S_1(i;\omega_1,\omega_2)
    =\ob{\prod_{i=1}^{\abs{\omega}}\beta_{\omega_{i}}\beta_{\omega_{i+1}}}
    v^{\sss (i)}_{\omega_{k+1}}u^{\sss (i)}_{\omega_{k}}v^{\sss
        (i)}_{\omega_{k}}u^{\sss (i)}_{\omega_{k-1}}\dots v^{\sss (i)}_{\omega_{2}}u^{\sss (i)}_{\omega_{1}}.
  \end{equation}
  Moving the final factor of $u^{\sss (i)}_{\omega_{1}}$ to the front contributes a
  factor $-1$ if and only if the colour $i\in \{1,2\}$ as in this case
  $u^{\sss (i)}_{\omega_{1}}$ is odd, while $u^{\sss
    (3)}_{\omega_{1}}$ is even. 
\end{proof}

An immediate consequence of this coloured graph representation is the
following colourless variant. For a directed cycle $C$ define
\begin{equation}
  \label{eq:wt-cycle}
  w(C) \bydef \prod_{xy\in C}\frac{\beta_{xy}}{r_{x}} = \prod_{xy\in C}
  \frac{\beta_{xy}}{m_{x}^{2}+\sum_{s}\beta_{xs}},
\end{equation}
and let $\cc L$ denote the collection of graphs whose connected components are
all directed cycles. Equivalently, will may think of an element
$L\in \cc L$ as a set of disjoint directed cycles.
\begin{corollary}
  \label{cor:Z}
  The partition function $Z$ can be expressed as
  \begin{equation}
    \label{eq:Z-cycle}
    Z = \ob{\prod_{s\in \mathcal V}r_{s}^{3}}\sum_{L\in \cc L}(-1)^{\abs{L}}\prod_{C\in L}w(C).
  \end{equation}
\end{corollary}
\begin{proof}
  This follows from Proposition~\ref{prop:Z-graph} by summing
  over the possible colours of each component, which yields a weight
  of $1-1-1=-1$ for each directed cycle. The vertex weights follow
  since every vertex in a cycle is also in exactly two coloured self-loops,
  and every vertex not in a cycle is in exactly three coloured self-loops. 
\end{proof}

\subsection{Computing $U(a,b,c)$, I}
\label{sec:computing-ux-y}

In this section we compute a graphical formula for the numerator
of~\eqref{eq:obs}, i.e., for the Grassmann integral of
\begin{equation}
  \label{eq:num-integrand}
  u^{\sss (2)}_{c}v^{\sss (2)}_{b} u^{\sss (3)}_{b}v^{\sss (3)}_{b} u^{\sss (1)}_{b}v^{\sss (1)}_{a}
  e^{S},
\end{equation}
when $a,b,c\in\cc V$ are all distinct. 
Towards this goal, we introduce two weighted sums of coloured graphs,
where the weight $w_{0}(G)$ is given by~\eqref{eq:Z-weight}: 
\begin{equation}
  \label{eq:GammaTheta}
  \Gamma(a,b,c)\bydef \sum_{G}\frac{w_{0}(G)}{r_{b}}, \qquad \Theta(a,b,c) \bydef
  \sum_{G}\sum_{C} \frac{w_{0}(G\cup C)}{r_{b}}.
\end{equation}
We will see in the following that the main contribution comes from
$\Gamma$, and $\Theta$ will be an error term.  The sum defining
$\Gamma(a,b,c)$ is a sum over coloured subgraphs $G$ which are unions
of a coloured self-avoiding walk $\gamma$ and coloured directed cycles
subject to the following conditions, in which $V(\gamma)$ denotes the
vertices contained in $\gamma$:
\begin{enumerate}
\item $\gamma$ is a self-avoiding walk from $a$ to $c$ that passes
  through $b$, of colour $1$ from $a$ to $b$, and of colour $2$ from
  $b$ to $c$,
\item the directed cycles are pairwise disjoint,
\item directed cycles of colour 1 do not intersect
  $V(\gamma) \setminus \{c\}$,
\item directed cycles of colour 2 do not intersect
  $V(\gamma) \setminus \{a\}$,
\item directed cycles of colour $3$ do not intersect
  $V(\gamma)\setminus \{a,c\}$, and contain at most one of $a$ and $c$.
\end{enumerate}
The sum defining $\Theta$ is a sum over coloured directed cycles $C$
of colour $3$ that contain both $a$ and $c$ and coloured graphs $G$
such that the coloured graph $C \cup G$ satisfies (i)--(iv) above and
\begin{itemize}
\item[(v')] directed cycles of colour $3$ do not intersect
  $V(\gamma)\setminus \{a,c\}$.
\end{itemize}

\begin{prop}
  \label{prop:num}
  For distinct $a,b,c\in \cc V$,
  \begin{equation}
    \label{eq:num}
    \int
    u^{\sss (2)}_{c}v^{\sss (2)}_{b} u^{\sss (3)}_{b}v^{\sss (3)}_{b} u^{\sss (1)}_{b}v^{\sss (1)}_{a}
    e^S
    = \Gamma(a,b,c) + \Theta(a,b,c).
  \end{equation}
\end{prop}

\begin{proof}[Proof of Proposition~\ref{prop:num}]
  Recall $e^{S}$ has the product form given by
  Lemma~\ref{lem:expS}. Our first step is to expand this product and
  see which terms combine with the prefactor
  in~\eqref{eq:num-integrand} to give a top-degree term. As in the
  proof of Proposition~\ref{prop:Z-graph} we characterize the
  contributing coloured graphs up to self-loops. After doing this we
  will justify the claimed weight.

  We first claim that, for each vertex $y\notin \{a,b,c\}$,
  \begin{enumerate}
  \item the coloured in- and out-degrees of $y$ are the same in any coloured
    graph $G$ with non-zero weight, and
  \item if the total in-degree of $y$ is two or three, then the weight is zero.
  \end{enumerate}
  The proof of these claims is exactly as in the proof of
  Proposition~\ref{prop:Z-graph} (this proof only used the
  combinatorics of the terms that arise from expanding $e^{S}$).

  We next consider the vertices $a,b,c$:
  \begin{enumerate}
  \item $a$ must have colour $1$ out-degree one and in-degree zero; $c$ must have colour
    $2$ in-degree one and out-degree zero; $b$ must have colour $1$ in-degree one, colour
    $2$ out-degree one, and colour $3$ in- and out-degree zero.
  \item Excepting the above points, the coloured in- and out-degrees
    of $a$ and $c$ must be equal,
   \item $a$ cannot have out-degree three, and $c$ cannot have
     in-degree three. 
  \end{enumerate}
  The factor
  $u^{\sss (2)}_{c}v^{\sss (2)}_{b} u^{\sss (3)}_{b}v^{\sss (3)}_{b}
  u^{\sss (1)}_{b}v^{\sss (1)}_{a}$ enforces the first claim: for
  example, $a$ must have colour $1$ out-degree one since the only
  terms in $e^{S}$ that contain $u_{a}^{\sss (1)}$ but not
  $v_{a}^{\sss (1)}$ are of the form $S_{1}(1;a,\cdot)$. The claims
  for $b$ and $c$ follow similarly. The second claim follows since
  self-loops contribute in- and out-degree one. The last claim follows
  by arguing as in the proof of Proposition~\ref{prop:Z-graph}; we
  give the argument for $a$ (it is similar for $c$). If $a$ had
  out-degree three, then it must have in- and out-degree one for
  colours $2$ and $3$. Let $\V{y}$ be the four-tuple of vertices to
  which $a$ is connected, by colour $2$ to $y_{1},y_{2}$, and colour
  $3$ to $y_{3},y_{4}$. There are two ways this can occur:
  $S_{1}(2;a,y_{1}) S_{1}(2;a,y_{2}) S_{1}(3;a,y_{3})
  S_{1}(3;a,y_{4})$ or $-S_{4}(2,3;a,\V{y})$, and these sum to zero.

  The preceding facts establish the claimed coloured graph
  structure. All that remains is to justify the weight. This is nearly
  as in the proof of Proposition~\ref{prop:Z-graph}: each self-loop at
  $y\in \cc V$ contributes a factor $r_{y}$, and each cycle coloured
  $1$ or $2$ contributes a factor $-1$. Note that there is never a
  self-loop at $b$ of colour 3 due to the factor
  $u^{(3)}_{b}v^{(3)}_{b}$. Up to the sign arising from the edges in
  $\gamma$, this gives the weight $r_{b}^{-1}w_{0}(G)$. All that
  remains is to show that the self-avoiding walk $\gamma$ from $a$ to
  $c$ does not contribute a sign. If $k^{\star}$ is the index such
  that $\omega_{k^{\star}}=b$, then this follows from observing that
  the weight of the walk has the form
  \begin{equation}
    \underbrace{u^{\sss (2)}_{\omega_{k}}v^{\sss (1)}_{\omega_{1}}}_{\mathrm{I}} v^{\sss
      (2)}_{\omega_{k}}u^{\sss (2)}_{\omega_{k-1}} v^{\sss
      (2)}_{\omega_{k-1}}u^{\sss (2)}_{\omega_{k-2}}\dots v^{\sss
      (2)}_{\omega_{k^{\star}+1}} u^{\sss (2)}_{\omega_{k^{\star}}} \underbrace{v^{\sss
      (2)}_{\omega_{k^{\star}}}u^{\sss (1)}_{\omega_{k^{\star}}}}_{\mathrm{II}}
    v^{\sss (1)}_{\omega_{k^{\star}}} u^{\sss
      (1)}_{\omega_{k^{\star}-1}}\dots v^{\sss (1)}_{\omega_{2}}
    u^{\sss (1)}_{\omega_{1}},
  \end{equation}
  where the factors labelled $\mathrm{I}$ and $\mathrm{II}$ come from
  the prefactor in the integrand. Moving the initial factor of
  $v^{\sss (1)}_{\omega_{1}}$ through to the end proves the claim.
\end{proof}

Given a self-avoiding walk $\gamma$, let $\cc L_{\gamma}$ be the
collection of sets of pairwise disjoint oriented cycles that are all vertex-disjoint
from $\gamma$.  Further, define
\begin{equation}
  \label{eq:saw-weight}
  w(\gamma) \bydef
  \prod_{i=1}^{\abs{\omega}-1}\frac{\beta_{\omega_{i}\omega_{i+1}}}{r_{\omega_{i}}}
  = \prod_{i=1}^{\abs{\omega}-1}
  \frac{\beta_{\omega_{i}\omega_{i+1}}}{m^{2}_{\omega_{i}}+
    \sum_{y\in\cc V}\beta_{\omega_{i}y}}.
\end{equation}

\begin{prop}
  \label{cor:num}
  For any three distinct vertices $a,b,c$, 
  \begin{equation}
    \label{eq:num1}
    \Gamma(a,b,c)
  = 
  \frac{\prod_{x}r_{x}^{3}}{r_{c}r_{b}^{2}} \sum_{\gamma\in
    \saw(a,c)}\sum_{L\in\cc L_{\gamma}}
  (-1)^{\abs{L}}w(\gamma)  1_{b\in \gamma}\prod_{C\in L}w(C),
  \end{equation}
  where $\saw(a,c)$ denotes the set of self-avoiding walks from
    $a$ to $c$.
\end{prop}
\begin{proof}
  The identity follows from the definition of $\Gamma$ by summing over
  the possible colourings. The self-avoiding walk must be
  vertex-disjoint from all directed cycles since any cycle
  intersecting exactly one endpoint can only have two possible
  colours, one of which is 3, and hence has total weight $1-1=0$.

  What remains is to explain the prefactor.  Recall the definition
  \eqref{eq:Z-weight} of $w_0$. The first product
  $\prod_{xy\in G} \beta_{xy}$ contains exactly those $\beta_{xy}$
  that appear in $w(\gamma)$ and in $\prod_C w(C)$. We are left with
  treating the $r$ factors, that is, for each vertex $x$ we should
  count the numbers of colours $i$ for which $(x,i)$ is not present in
  $G$.

  For $x\notin G$ there is a contribution of $r_x^3$. If
  $x\in G\setminus \{a,b,c\}$, it has a single incoming edge and a
  single outgoing edge, both of the same colour, leaving two colours
  that are not present. This is a contribution of $r_x^2$. Since the
  weight $w$ (either $w(\gamma)$ if $x\in\gamma$ or $w(C)$ if
  $x\in C$) contains a factor $r_x^{-1}$ we are left with the
  prefactor $r_x^{3}$.

  The vertex $a$ has exactly one outgoing edge of colour $1$, leading
  to a contribution of $r_a^{2}$. Since $w(\gamma)$ contains a factor
  $r_a^{-1}$ we are left with the prefactor $r_a^3$. The vertex $b$
  has adjacent edges of colours $1$ and $2$ but not of colour
  $3$. Hence only $(b,3)$ is not present in $G$, and the corresponding
  factor is $r_b$. Since $w(\gamma)$ contains a factor $r_b^{-1}$ and
  in the definition~\eqref{eq:GammaTheta} of $\Gamma$ there is an
  additional $r_b^{-1}$, we are left with the prefactor
  $r_b$. Finally, for the vertex $c$, both $(c,1)$ and $(c,3)$ are not
  present in $G$, yielding a factor $r_c^2$. Since $\gamma$ does not
  take any steps from $c$, there is no additional factor coming from
  $w(\gamma)$, and we are left with $r_c^2$.
\end{proof}

\subsection{Heaps of pieces and loop-erased walk}
\label{sec:HoP}

An important combinatorial structure, closely related to loop erased
random walks, is a \emph{heap of cycles}. This was initially studied
in \cite{Viennot} as an application of Viennot's more general theory
of heaps of pieces. We will briefly recall the definitions and results
that we need; an enjoyable introduction to the theory can be found
in~\cite{Krattenthaler}.
\begin{defn}
  A \emph{labeled heap of cycles} on $\mathcal G$ is a partially
  ordered set $(\cc X,\preceq)$, each of whose elements
  $\alpha \in \cc X$
  is assigned a cycle $C_\alpha$, such that for every $\alpha,\beta
  \in \cc X$:
  \begin{enumerate}
  \item if $C_\alpha$ intersects $C_\beta$ then $\alpha$ and $\beta$
    are comparable; and
  \item if $\alpha\preceq\beta$ and $\beta$ covers $\alpha$
    ($\alpha\preceq \gamma\preceq \beta \Rightarrow$
    $\gamma\in \{\alpha,\beta\}$) then $C_\alpha$ intersects
    $C_\beta$.
  \end{enumerate}
  Two labeled heaps of cycles are isomorphic if they are isomorphic as
  ordered sets, with an isomorphism which respects the cycle
  assignment.  A \emph{heap of cycles} is an equivalence class (under
  isomorphism) of labeled heaps of cycles.
\end{defn}
In the sequel we will be slightly informal and will refer to the
cycles in a heap of cycles, as opposed to speaking of the labels of the
elements of a labelled heap of cycles. We will be interested in
weighted heaps of cycles. To this end, let $\{q_e\}_{e\in\mathcal{E}}$
be a set of formal commuting variables. Define, for each subgraph
$G=(V(G),E(G))$ of $\mathcal{G}$, a weight
$q(G)=\prod_{e\in E(G)} q_e$. For a heap of cycles $L$ we define
$q(L)=\prod_{C\in L}q(C)$.

We now come to the connection between heaps of cycles and loop-erased
walk. Consider a walk $\omega$ from $x$ to $z$. By performing the
loop-erasure process, we obtain a self-avoiding path
$\gamma=\LE(\omega)$, and a set of erased cycles. In order to
reconstruct $\omega$ from $\gamma$ and this set of cycles, however, we
need some extra information on the order in which the cycles were
erased. The essence of~\cite[Proposition 6.3]{Viennot}, due to
Viennot, is that this additional information is, in fact, the heap of
cycles structure. Formally, for $V\subset \cc V$, let $\mathcal{H}_V$
denote the set of heaps of cycles whose maximal elements intersect
$V$. 
\begin{theorem}
  \label{thm:viennot}
  Fix a self-avoiding walk $\gamma$. There is a weight-preserving
  bijection between $\mathcal{H}_\gamma$ and the set of paths whose
  loop erasure is $\gamma$, weighted by their loops. In particular,
  \begin{equation}\label{eq:viennot}
    \sum_{\omega\colon \LE(\omega)=\gamma}q(\omega) =
    q(\gamma)\sum_{L\in\mathcal{H}_\gamma}q(L). 
  \end{equation}
\end{theorem}
This theorem was stated without proof
  as~\cite[Proposition~6.3]{Viennot}; we remark that proofs can be
  found in~\cite{Marchal} or~\cite[Appendix A]{HelmuthLWW}.

Recall that $\mathcal L$ denotes the set of all collection of
non-intersecting cycles. Equivalently, $\mathcal L$ is the set of
heaps of cycles with empty order; sometimes these are called
\emph{trivial} heaps.  For any set $V\subset \cc V$ of vertices let
$\mathcal{L}_V$ denote that set of trivial heaps of cycles in
$\mathcal L$ that do not contain any cycle intersecting $V$. We will
also make use of the following general identity, which is the basic
connection between the graph-theoretic objects of the previous section
and walks.
\begin{prop}[Proposition 5.3 of \cite{Viennot}] \label{prop:heapsofpieces}
For every $V\subset \mathcal{V}$, as formal series in $\{q_{e}\}$,
\begin{equation}
  \frac{\sum_{L\in \mathcal{L}_V} (-1)^{|L|} q(L)}{\sum_{L \in \mathcal{L}} (-1)^{|L|} q(L)} = \sum_{L \in \mathcal{H}_V} q(L).
\end{equation} 
\end{prop}
We remark that while the preceding proposition is stated without
proof~\cite[Proposition~5.3]{Viennot}, the proof follows from the
methods therein.

\subsection{Computing $U(a,b,c)$, II}
\label{sec:loop-erased-walk}

Let $\sum_{\omega\colon a\to c}$ denote the sum over all walks
  beginning at $a$ and ending at $c$. Recall from~\eqref{eq:obs} that
$C_{bc}=m_{c}^{2}r_{b}^{2}$. 
\begin{prop}
\label{cor:gamma}
For any three distinct vertices $a,b,c\in\cc V$,
\begin{equation}
\label{eq:gamma}
\frac{C_{bc}\Gamma(a,b,c)}{Z} =
\sum_{\omega\colon a \to c} 1_{b \in \LE(\omega)} \P_a[(X_n)_{n<h_\blacktriangle}=\omega].
\end{equation}
\end{prop}
\begin{proof}
By Proposition~\ref{cor:num} and Corollary~\ref{cor:Z}, 
setting $q_{xy} = \beta_{xy} r_x^{-1}$,
\begin{align}
\frac{\Gamma(a,b,c)}{Z} &= r_b^{-2}r_c^{-1} \sum_{\gamma \in
  \saw(a,c)} q(\gamma)1_{b\in\gamma} \frac{\sum_{L\in
    \mathcal{L}_\gamma} (-1)^{|L|} q(L)}{\sum_{L\in\mathcal{L}}
  (-1)^{|L|} q(L)} \\
&= 
r_b^{-2}r_c^{-1} \sum_{\gamma \in \saw(a,c)}q(\gamma)1_{b\in\gamma}
  \sum_{L\in \mathcal{H}_\gamma} q(L);
\end{align}
using Proposition~\ref{prop:heapsofpieces} for the last equality. 
We can now apply Theorem \ref{thm:viennot}, obtaining
\begin{equation}
\frac{\Gamma(a,b,c)}{Z} = r_b^{-2}r_c^{-1} \sum_{\omega\colon a \to c}
1_{b\in \LE(\omega)} q(\omega)
\end{equation}
as formal power series. The result holds in the sense of convergent
power series when interpreting $q_{xy}$ numerically as
$\beta_{xy}r_{x}^{-1}$, as after multiplying by $C_{bc}$ we recognize
the right-hand side to be the right-hand side of~\eqref{eq:gamma}
since $m_{c}^{2}r_{c}^{-1}$ is the probability of a jump from $c$ to
$\blacktriangle$.
\end{proof}

Next we show that the term $\Theta$ is negligible compared to $\Gamma$
as $m_{c}^{2}\to \infty$.
\begin{prop}
  \label{prop:thetasmall}
For any three distinct vertices $a,b,c\in\cc V$
\begin{equation} \label{eq:theta}
\left|\frac{\Theta(a,b,c)}{Z}\right| \le \frac{\Gamma(a,b,c)}{Z} \cdot
\P_{a}[X_{h^{-}_{\blacktriangle}}=c] \cdot
\frac{\sum_{v\in \cc V}\beta_{cv}}{m_{c}^{2}}
\max_{v\in\mathcal{V}}{\E_{v} [h^{-}_{\blacktriangle}]}.
\end{equation}
\end{prop}
\begin{proof}
  In each term of $\Theta$ we identify three coloured self-avoiding
  walks $\gamma_{i}$. The first goes from $a$ to $c$ passing through
  $b$, with colour 1 from $a$ to $b$ and colour 2 from $b$ to $c$. The
  second goes from $a$ to $c$ and is of colour 3, and the third goes
  from $c$ to $a$ and is of colour 3. Moreover, the $\gamma_{i}$
  intersect only at $a$ and $c$. Let $V(\gamma_{i})$ denote the
  vertices in $\gamma_{i}$, and let
  $V=V(\gamma_1)\cup V(\gamma_2) \cup V(\gamma_3)$. Then, summing over
  the colouring of the other edges as in the proof of Proposition~\ref{cor:num},
\begin{equation}
\Theta(a,b,c) = \frac{\prod_{x\in \cc V} r_x^3}{r_{c}r_{b}^{2}}
\mathop{\sum\nolimits^{'}}_{\gamma_1,\gamma_2,\gamma_3} \sum_{L \in
  \mathcal{L}_{V}} q(\gamma_1) q(\gamma_2) q(\gamma_3) (-1)^{|L|}
q(L), 
\end{equation} 
with $q_{xy}=\beta_{xy} r_y^{-1}$. The notation
$\sum'_{\gamma_{1},\gamma_{2},\gamma_{3}}$ means that the
self-avoiding walks $\gamma_{i}$ intersect only at $a$ and $c$. The
$(-1)^{|L|}$ factor comes from the fact that each loop in $L$ is
counted with weight $-1$ twice (for colouring 1 and 2) and once with
weight +1 (for colouring 3).

By Proposition~\ref{prop:heapsofpieces} and Corollary~\ref{cor:Z}, as in the proof of Proposition~\ref{cor:gamma},
\begin{equation}
\label{eq:thetaheap}
\frac{\Theta(a,b,c)}{Z} = \frac{1}{r_{c}r_{b}^{2}}\mathop{\sum\nolimits^{'}}_{\gamma_1, \gamma_2, \gamma_3} q(\gamma_1)q(\gamma_2)q(\gamma_3) \sum_{L\in \mathcal{H}_V}q(L).
\end{equation}

Controlling this term requires a small digression.  Let $L_1$ and
$L_2$ be two heaps of cycles. The \emph{superposition $L_1 \odot L_2$
  of $L_2$ on $L_1$} is the heap of cycles that results from putting
$L_{2}$ on top of $L_{1}$. More formally, the elements of
$L_{1}\odot L_{2}$ are the disjoint union of the elements of $L_1$ and
those of $L_2$, the cycle assignment remains the same, and the order
is the transitive closure of the following order: fix elements
$\alpha_{1}$ and $\alpha_{2}$ such that $C_{\alpha_{1}}$ intersects
$C_{\alpha_{2}}$. If both elements belong to $L_1$ (respectively, $L_2$) they
keep the order they had in $L_1$ (respectively, $L_2$). If, on the
other hand, $\alpha_{1}\in L_1$ and $\alpha_{2} \in L_2$, then
$\alpha_{1}<\alpha_{2}$ (and vice versa). Superposition is an
associative binary operation, and $q$ is a homomorphism for this
algebraic structure, i.e., $q(L_1 \odot L_2) = q(L_1) q(L_2)$,
see~\cite{Viennot}.

For any two sets $V_1,V_2$, the superposition $L=L_1 \odot L_2$ of a
heap of cycles $L_2 \in \mathcal{H}_{V_1}$ on a heap of cycles
$L_1 \in \mathcal{H}_{V_2}$ is in $\mathcal{H}_V$ for
$V=V_1 \cup V_2$. Moreover, this map is surjective. That is, any heap
of cycles $L\in \mathcal{H}_V$ can be expressed as $L=L_1 \odot L_2$
for some $L_1 \in \mathcal{H}_{V_1}$ and $L_2 \in
\mathcal{H}_{V_2}$. To see this, given a heap in $\mathcal{H}_{V}$,
define $L_1$ to consist of the cycles in $L$ that intersect $V_{1}$ or
are smaller than a cycle intersecting $V_1$, let $L_2$ will consist of
the other elements of $L$. It is then straightforward to verify that
$L_1 \in \mathcal{H}_{V_1}$, $L_2 \in \mathcal{H}_{V_2}$, and
$L=L_1 \odot L_2$.

We now return to estimating $\frac{\Theta(a,b,c)}{Z}$. We continue to
write $q_{xy}$ in place of $\beta_{xy}r_{x}^{-1}$ for brevity, but to
make estimates we are considering series as analytical objects. Since
the weights are positive, \eqref{eq:thetaheap} and the surjectivity of
superposition implies
\begin{equation}
\left|\frac{\Theta(a,b,c)}{Z} \right| \le \frac{1}{r_{c}r_{b}^{2}}\mathop{\sum\nolimits^{'}}_{\gamma_1, \gamma_2, \gamma_3} q(\gamma_1)q(\gamma_2)q(\gamma_3)
\sum_{L_1\in \mathcal{H}_{\gamma_1}} \sum_{L_2\in \mathcal{H}_{\gamma_2}} \sum_{L_3\in \mathcal{H}_{\gamma_3}}q(L_1)q(L_2)q(L_3),
\end{equation}
where the right-hand side may be infinite.

We now bound this expression by replacing $\sum^{'}$ by $\sum$, i.e.,
by relaxing the constraint that the $\gamma_{i}$ are mutually
self-avoiding. This results in a product of three terms -- a
self-avoiding walk $\gamma_1$ from $a$ to $c$ passing through $b$, and
two self-avoiding walks $\gamma_2$ and $\gamma_3$, one from $a$ to $c$
and the other from $c$ to $a$. The first term is, by following the
proof of Corollary~\ref{cor:gamma},
\begin{equation}
  \label{eq:Theta1}
  \frac{1}{r_{c}r_{b}^{2}}\sum_{\gamma_1}q(\gamma_1)\sum_{L_1\in \mathcal{H}_{\gamma_1}}
q(L_1) = \frac{\Gamma(a,b,c)}{Z}.
\end{equation}

The second term,
$\sum_{\gamma_2}
\sum_{L_2\in\mathcal{H}_{\gamma_2}}q(\gamma_2)q(L_2)$, can be
analyzed in the same fashion. By Theorem \ref{thm:viennot},
\begin{equation}
\sum_{\gamma\colon a \to c} \sum_{L\in\mathcal{L}_{\gamma}}q(\gamma)q(L) =
\sum_{\omega\colon a \to c} q(\omega) = 
 \frac{m_c^2 + \sum_{v\in\cc V} \beta_{cv}}{m_c^2} ~ \P_a\cb{X_{h^{-}_{\blacktriangle}}=c}.
\end{equation} 

The third term,
$\sum_{\gamma_3}
\sum_{L_3\in\mathcal{H}_{\gamma_3}}q(\gamma_3)q(L_3)$, counts all
walks from $c$ to $a$ by Theorem~\ref{thm:viennot}. The first step, to
some $x$ such that $\beta_{cx}>0$, has weight
$\frac{\beta_{cx}}{r_c}$, so
\begin{equation}
\sum_{\gamma\colon c \to a}
\sum_{L\in\mathcal{H}_{\gamma}}q(\gamma)q(L) 
= \sum_{x} \frac{\beta_{cx}}{r_c} \sum_{\omega\colon x \to a} q(\omega) =
\sum_{x} \frac{\beta_{cx}}{r_c} \sum_{k\in \N} 
\sum_{\substack{\omega\colon x \to a \\ |\omega|=k}} \!\! \P_x[(X_n)_{n\le k}=\omega].
\end{equation}
The internal sum on the right-hand side is $\P_x[h_\blacktriangle > k]$, and
summing this over $k$ yields $\E_x[h_\blacktriangle^{-}]$. Optimizing the upper bound over $x$ gives
\begin{equation}
  \sum_{\gamma_3}
\sum_{L_3\in\mathcal{H}_{\gamma_3}}q(\gamma_3)q(L_3) \leq
\frac{\sum_{x\in\cc V} \beta_{cx}}{m_c^2 + \sum_{v\in\cc V} \beta_{cv}} ~ \max_{x\in\mathcal V}{\E_s[h_\blacktriangle]}.
\end{equation}
Taking the product of all three terms yields \eqref{eq:theta}.
\end{proof}

\subsection{Proof of Theorem~\ref{thm:main}}
\label{sec:mainthm}
\begin{proof}[Proof of Theorem~\ref{thm:main}]
  This follows from Proposition~\ref{prop:num} combined with
  Propositions~\ref{cor:gamma} and~\ref{prop:thetasmall}.
\end{proof}

\section{Proof of Theorem~\ref{thm:main-2}}
\label{sec:thm2}

In this section we prove Theorem~\ref{thm:main-2}; since the argument
are rather similar to those of the preceding section we will be a
little brief in some places.

\begin{lemma}\label{lem:expS'}
  The exponential of the action $S'$ from~\eqref{eq:action-2} is given by
  \begin{equation}\label{eq:expS'}
    e^{S'} = \prod_{x,i} \left(1+r_x u^{\sss (i)}_x v^{\sss (i)}_x \right) \,
    \prod_x \left( 1 + \sum_{i,y} \beta_{xy}v^{\sss (i)}_y u^{\sss (i)}_x \right).
  \end{equation}
\end{lemma}
\begin{proof}
 Using $\log(1+x)=x-x^{2}/2+x^{3}/3+\dots$, the logarithm
  of~\eqref{eq:expS'} is
  \begin{equation}
    \sum_{x,i}r_{x}u_{x}^{\sss (i)}v_{x}^{\sss (i)} +
    \sum_{x}\sum_{y,i}\beta_{xy}v_{y}^{\sss (i)}u_{x}^{\sss (i)} -
    \frac{1}{2}\sum_{x}\ob{\sum_{y,i}\beta_{xy}v_{y}^{\sss  (i)}u_{x}^{\sss (i)}}^{2} 
    +\frac{1}{3}\sum_{x}\ob{\sum_{y,i}\beta_{xy}v_{y}^{\sss (i)}u_{x}^{\sss (i)}}^{3}, 
  \end{equation}
  since any higher order terms include a factor of
  $(u_{x}^{\sss (i)})^{2}=0$ for some $i$.  This establishes the
  result.
\end{proof}

\begin{prop}
  \label{prop:Z'}
  Let $Z'\bydef \int e^{S'}$. Then $Z=Z'$.
\end{prop}
\begin{proof}
  We represent the terms in the expansion of the products
  in~\eqref{eq:expS'} as coloured subgraphs of $\cc G$ with self-loops
  by viewing each monomial $v^{\sss (i)}_{y} u^{\sss (i)}_{x}$ as a
  directed edge with colour $i$ from $x$ to $y$, and each monomial
  $u^{\sss (i)}_x v^{\sss (i)}_x$ as a self-loop of colour $i$. We
  recall that $\beta_{xx}=0$ for all
  $x$.

  By definition, $\int e^{S}$ is the coefficient of the top degree
  term of $e^{S}$. This coefficient is a sum over coloured directed
  graphs $G$ with self-loops such that the coloured in- and out-degree
  of every vertex is one for all colours.

  We first consider which graphs $G$ without self-loops that result
  from expanding the second product in~\eqref{eq:expS'} can make a non-zero
  contribution. We claim that a non-zero contribution can only arise
  if for each vertex $x$, $x$ has either in- and out-degree zero, or
  if $x$ has colour in- and out-degree one for exactly one colour. 
  
  To see this, note that a self-loop of colour $i$ at $x$ contributes
  in- and out-degree one of colour $i$ at $x$. Hence if $G$ had a
  vertex with unequal coloured in- and out-degrees, the term cannot be
  top-degree. Moreover, the form of the product in~\eqref{eq:expS'}
  ensures each vertex has out-degree at most one. This proves the
  claim, and it follows that the contributing $G$ are unions of
  pairwise disjoint monochromatic directed cycles.
  
  To conclude the proof, we note that the coefficient in front of a
  term represented by such a graph $G$ is equal $w_0(G)$, just as in
  Proposition \ref{prop:Z-graph}.
\end{proof}

We next consider the numerator of $U'(a,b,c)$ as defined in~\eqref{eq:U-2}. Let
\begin{equation}
  \label{eq:gamma'}
  \Gamma'(a,b,c) \bydef \sum_G w_0(G).
\end{equation}
where the sum defining $\Gamma'$ is a sum over coloured subgraphs $G$
which are unions of a coloured self-avoiding walk $\gamma$ and coloured
directed cycles subject to the following conditions, in which
$V(\gamma)$ denotes the vertices contained in $\gamma$:
\begin{enumerate}
\item $\gamma$ is a self avoiding walk from $a$ to $c$ that passes
  through $b$, of colour $1$ from $a$ to $b$, and of colour $2$ from
  $b$ to $c$,
\item the directed cycles are pairwise disjoint,
\item directed cycles of colour $1$ do not intersect $V(\gamma)\setminus \{c\}$, 
\item directed cycles of colour $2$ do not intersect $V(\gamma)$, 
\item directed cycles of colour $3$ do not intersect $V(\gamma)\setminus \{c\}$.
\end{enumerate}

\begin{prop}
  \label{prop:num'}
  For distinct $a,b,c \in \mathcal V$,
  \begin{equation}
    \label{eq:num'}
    \int u^{\sss (2)}_c v^{\sss (2)}_b u^{\sss (1)}_b v^{\sss (1)}_a e^{S'} = \Gamma'(a,b,c).
  \end{equation}
\end{prop}

\begin{proof}
  We will expand the product \eqref{eq:expS'}, and see which terms
  combine with the prefactor in \eqref{eq:num'} to give a top-degree
  term. As in the proof of Proposition~\ref{prop:Z'}, we will first
  characterize the subgraphs of $\mathcal G$ corresponding to these
  terms.

  Assume that $G$ is such a graph.  We first consider a vertex
  $x\notin \{a,b,c\}$. Then, just as in the proof of
  Proposition~\ref{prop:Z'}, either $x$ has in- and out-degree zero,
  or $x$ has colour in- and out-degree one for exactly one colour.

  Next we consider the vertices $\{a,b,c\}$, beginning with $a$. Due
  to the term $v^{\sss (1)}_a$ in \eqref{eq:num'}, there cannot be a
  self-loop of colour $1$ at $a$. Therefore, since $G$ corresponds to
  a top-degree term, $u^{\sss (1)}_a$ must appear in the expansion
  of the second term in \eqref{eq:expS'}. That is, $a$ has an
  outgoing edge of colour $1$, and no other outgoing edges. Moreover,
  the terms $u^{\sss (2)}_a$ and $u^{\sss (3)}_a$ must appear as
  self-loops. Thus $G$ contains a single edge adjacent to $a$, which
  is outgoing of colour $1$.

  For $b$, the term $v^{\sss (2)}_b$ implies $b$ has a single outgoing
  edge, of colour $1$. The term $u^{\sss (1)}_b$ implies $b$ has an
  incoming edge of colour $2$. Since $u^{\sss (3)}_b$ can only appear
  as a self loop, we conclude that $b$ has in- and out- degree zero of
  colour $3$.

  For $c$, the term $u^{\sss (2)}_c$ implies $c$ has in-degree one and
  out-degree zero for colour $2$. An argument as in the proof of
  Proposition~\ref{prop:Z'} shows $c$ can either have (i) no adjacent edges of
  colours $1$ and $3$, or (ii) in- and out-degree one exactly one of these
  colours. 

  These observations imply that the graph indeed has the structure
  given in the definition of $\Gamma'$. The weight can be calculated
  as in the proof of Proposition~\ref{prop:num}.
\end{proof}

\begin{prop}
\label{claim:num'}
For any three distinct vertices $a,b,c\in \cc V$,
$\Gamma'(a,b,c)=r_{b}^{-1} \Gamma(a,b,c)$, i.e.,
\begin{equation}
\Gamma'(a,b,c) = \frac{\prod_x r_x^3}{r_c r_b} \sum_{\gamma \in \saw(a,c)} \sum_{L\in\mathcal{L}_\gamma}(-1)^{|L|}w(\gamma)1_{b\in\gamma} \prod_{C\in L}w(C).
\end{equation}
\end{prop}
\begin{proof}
  As is the proof of Proposition~\ref{cor:num}, we sum the possible
  colourings in~\eqref{eq:gamma'}.  Note that graphs that intersect
  $\gamma$ at the vertex $c$ appear in~\eqref{eq:gamma'} have total
  weight $0$, since each such graph appears once when the cycle
  intersecting $c$ is of colour $1$ and once when it has colour $3$,
  and these two coloured graph have opposite weights.

  The vertex factor is calculated in the same manner as in
  Proposition~\ref{cor:num}.
\end{proof}

\begin{proof}[Proof of Theorem~\ref{thm:main-2}]
  By the definition~\eqref{eq:U-2} of $U'(a,b,c)$ and
  Propositions~\ref{prop:Z'}--\ref{claim:num'},  
  \begin{equation*}
    U'(a,b,c) = \frac{C'_{bc} \Gamma'(a,b,c)}{Z'} = \frac{C_{bc}\Gamma(a,b,c)}{Z},
  \end{equation*}
  and hence the Theorem follows by applying Proposition~\ref{cor:gamma}.
\end{proof}

\bibliographystyle{abbrv}
\bibliography{refs}

\end{document}